 \documentclass[11pt]{article}
 \usepackage[auth-sc]{authblk}
\usepackage{amssymb}
\usepackage{amsthm}
\usepackage{graphicx}
\usepackage{latexsym}
\usepackage{amsmath,amscd}
\usepackage{bm}
\def\part#1{\frac{\partial\phantom{q}}{\partial#1}}

\newenvironment{remark}{\begin{trivlist}\item[]{\bf Remark:} }
{\end{trivlist}}

\newtheorem{thm}{Theorem}

\newtheorem{prp}[thm]{Proposition}

\newtheorem{corollary}[thm]{Corollary}

\def\End{\mathop{\rm End}\nolimits}
\def\Jac{\mathop{\rm Jac}\nolimits}

\def\rk{\mathop{\rm rk}\nolimits}
\def\deg{\mathop{\rm deg}\nolimits}
\def\tr{\mathop{\rm tr}\nolimits}
\def\Pic{\mathop{\rm Pic}\nolimits}
\newcommand{\C}{\mathbf{C}}
\newcommand{\Z}{\mathbf {Z}}
\newcommand{\CP}{{\mathbf C}{\rm P}}
\newcommand{\PP}{{\rm P}}
\newcommand{\OO}{{\mathcal{O}}}
\begin{document}



\title{Multiplicity algebras for rank 2 bundles on curves of  small genus} 
\author{Nigel Hitchin}

\maketitle

\begin{abstract}
In \cite{H1} Hausel introduced a commutative  algebra -- the {\it multiplicity algebra} -- associated to a fixed point of the $\C^*$-action on the Higgs bundle moduli space. Here we describe this algebra for a fixed point consisting of a very stable rank 2 vector bundle and zero Higgs field for a curve of low genus. Geometrically, the relations in the algebra are described by a family of quadrics and we focus on the discriminant of this family, providing  a new viewpoint on the moduli space of stable bundles. The discriminant  in our examples demonstrates that as the bundle varies, we obtain a   continuous variation in the isomorphism class of the algebra.
\end{abstract}

\vskip .5cm
\centerline{\it{Dedicated to Oscar Garc\'ia-Prada on the occasion of his 60th birthday} }

\section{Introduction}	
The integrable system introduced in \cite{Hit2} is now 35 years old but there are still unexplored features. A new issue arose in the recent paper \cite{HH} with  Tam\'as Hausel and work which followed on from it \cite{H1}.  It concerns fixed points $m$ of the $\C^*$-action on the moduli space ${\mathcal M}$ of Higgs bundles (these fixed point sets form a subject on which Oscar Garc\'ia-Prada has in particular made fundamental contributions   e.g.\cite{GP1},\cite{GP2},\cite{GP3}). Such a point is given by a Higgs bundle $(E,\Phi)$ where $\Phi$ is necessarily nilpotent and so the functions $h=(h_1,\dots,h_n)$ defining the integrable system all vanish. Then one defines from the $h_i$ a  commutative algebra associated to the fixed point $m$, which is called the {\it multiplicity algebra}  \cite{H1}. When the fixed point is {\it very stable}, the algebra is finite-dimensional and the dimension gives the multiplicity of the component of the nilpotent cone  containing $m$, hence the name.  An interesting observation is that 
for certain fixed points 
at the upper end of the nilpotent cone the multiplicity algebras are isomorphic to  the cohomology of  homogeneous spaces and it is this which suggests further research into their structure in more generality. 

 This paper is about the algebras defined at the {\it lower} end of the nilpotent cone, namely the case of a very stable bundle, where $\Phi=0$. Here we  learn about  the algebra by studying examples, but some underlying structure remains in general: in all cases the $\C^*$-action defines a grading and yields a Poincar\'e duality ring -- a nondegenerate pairing on  subspaces of complementary degree. 

When $E$ has rank $2$ and fixed determinant the relations for the algebra are given by $3g-3$ homogeneous quadratic functions on a vector space of dimension $3g-3$, more invariantly given by 
the quadratic map $\tr \Phi^2: H^0(C,\End_0 E\otimes K)\rightarrow H^0(C,K^2)$. Considering $H^0(C,\End_0 E\otimes K)$ as the cotangent space at the point $[E]$ of the moduli space ${\mathcal N}$ of stable bundles, this is the definition of the integrable system. 

The particular issue we address here is to exhibit algebras which, up to equivalence, contain continuous parameters unlike the integral structure of cohomology. We use the discriminant: a class $\alpha \in H^1(C,K^*)$ defines a homomorphism $\alpha: H^0(C,\End_0 E\otimes K)\rightarrow H^1(C,\End_0 E)$ which by Serre duality is equivalent to the quadratic map above. Then $\det \alpha=0$ describes a  hypersurface of degree $3g-3$ in the projective space $\PP(H^1(C,K^*))$ and the projective equivalence class of this is an invariant of the algebra, independent of any specific choice of generators and relations. In the terminology of classical algebraic geometry we have a family of quadrics and the discriminant is a familiar invariant. 

In genus $g=2$ we have three quadratic functions on a 3-dimensional vector space which geometrically is a net of conics. In \cite{H} we showed that if the degree of $E$ is odd then there are just finitely many equivalence classes of algebras, the very stable ones  having three generators $\xi_i$ with $\xi_i^2=0$. This is a case where the algebra is  indeed isomorphic to  the cohomology  $H^*((\CP^1)^3,\C)$.  In contrast we show in Section \ref{two}  that when $E$ has  even degree  the discriminant is a  cubic curve isomorphic to   a plane section of a singular cubic surface,  and the modulus of this elliptic curve varies as $E$ varies in the moduli space of stable bundles. To do this we make  essential use of the description in \cite{vGP}  of the integrable system using a classical relationship between certain curves of genus 2,3 and 5. The discriminant approach also gives a new viewpoint on the moduli space:  the space of discriminants can be seen as the projection from a point on the Igusa quartic threefold, the quotient of the moduli space $\PP^3$ of stable bundles (as described in the foundational paper \cite{Nar0}) by the group $H^1(C,\Z_2)$. 

The geometry of these same three curves comes into play when we consider in Section \ref{three}  a family of bundles of degree zero on a general nonhyperelliptic curve of genus $3$. The moduli space ${\mathcal N}$ for $\Lambda^2E\cong \OO$ is well-known  \cite{Nar1} to be a special singular quartic hypersurface in $\PP^7$, initially studied by A.Coble. We consider a two-dimensional family by choosing a line bundle $U$ of order $2$  and looking at the fixed point set in ${\mathcal N}$ under the action $E\mapsto E\otimes U$. This is a quartic surface in $\PP^3\subset \PP^7$, in fact the Kummer surface of the Jacobian for a genus $2$ curve, part of the story of the previous section. 

The results of Pauly \cite{P}  allow us to write down  the six relations for the algebra and to determine the discriminant variety more geometrically. We find that for our chosen family the degree 6 hypersurface is reducible to a singular quadric and a quartic. The degeneracy subspace of the quadric is a projective plane 
and its intersection with the quartic hypersurface is a quartic curve. It is the modulus of this curve which varies as we vary $E$ in the family, but we use a genericity argument to show this rather than using the explicit formula. The geometry behind this introduces a quartic surface  with 10 nodes associated to a pair of conics. 

The original motivation for this study in \cite{HH} involves mirror symmetry for the Higgs bundle moduli space. The precise role of the multiplicity algebra in this context has yet to be determined. 

\section{The quadratic map}
Let $C$ be a curve of genus $g>1$, $E$ a rank $2$ vector bundle over $C$ with $\Lambda^2 E$ fixed, and let $\End_0E$ denote the bundle of trace zero endomorphisms of $E$. For $\Phi\in H^0(C, \End_0 E\otimes K)$ we consider $\tr \Phi^2\in H^0(C,K^2)$. When $E$ is stable, then the dimension of both these spaces is $3g-3$ and we consider the map
 $$\tr \Phi^2:  H^0(C, \End_0 E\otimes K)\rightarrow  H^0(C,K^2)$$
which commutes with the scalar action of $\lambda\in \C^*$ on $\Phi$ and $\lambda^2$ on the right hand side. As $[E]$ varies over the moduli space 
${\mathcal N}$ of stable bundles, linear functions on $H^0(C,K^2)$ generate the integrable system of \cite{Hit2}. 

We can view the situation in different ways, one is to regard it as a family of quadrics in $\PP^{3g-4}\cong \PP(H^0(C, \End_0 E\otimes K))$ parametrized by $\PP^{3g-4}\cong \PP(H^1(C,K^*))$, evaluating $\alpha\in H^1(C,K^*)=H^0(C,K^2)^*$ on $\tr\Phi^2$. Thus in genus $2$ it is  a net of conics using the classical term. The locus of singular quadrics is a hypersurface, the {\it discriminant},  in   $\PP(H^1(C,K^*))$.

A bundle $E$ is called {\it very stable} \cite{Laum} if there is no nonzero nilpotent Higgs field $\Phi\in H^0(C, \End_0 E\otimes K)$. In rank 2 nilpotency is equivalent to $\tr \Phi^2=0$, or a non-empty base locus for the  family of quadrics. The very stable bundles are important, not only because they constitute the generic situation, but  also because then the quadratic map is proper \cite{PPN}, a crucial feature in the more general setting of \cite{HH}. 

One further point -- our question just involves the endomorphism bundle and not $E$ itself and so is insensitive to the action $E\mapsto E\otimes L$ where $L$ is a line bundle (with $L^2$ trivial if we want to fix $\Lambda^2E$). Equivalently, we are only concerned with the projective bundle $\PP(E)$.

\section{Genus two curves} \label{two} 
\subsection{Odd degree}
Let $C$ be a smooth curve of genus $2$, then Atiyah \cite{At} described the projective bundle associated to an indecomposable rank $2$ vector bundle $E$ of odd degree.  The construction represents the projective bundle by a vector bundle which is an 
 extension ${\mathcal O}\rightarrow E\rightarrow L$ where $\deg L=1$. 

The extension class $[\alpha]$ lies in  $H^1(C,L^*)$  and since $\dim H^1(C,L^*)=2$ it is uniquely determined up to a multiple by its annihilator, a section  $s\in H^0(C,KL)$ with $[\alpha] s=0\in H^1(C,K)$. Then $s$ has divisor 
$p+q+r$. 

Now $E$ has 4 degree zero subbundles, and the four divisors are 
$$ p+q+r,\quad p+\sigma(q)+\sigma(r),\quad  q+\sigma(r)+\sigma(p),\quad r+\sigma(p)+\sigma(q)$$ 
where 
 $\sigma:C\rightarrow C$ is the hyperelliptic involution.  Let $\pi:C\rightarrow \PP^1$ denote the quotient map, then $(\pi(p),\pi(q),\pi(r))=(a_1,a_2,a_3)$ and out of the eight inverse images, the four above give the same projective bundle. Hence we have Atiyah's  description of the moduli space as a branched double covering of the symmetric product $S^3\PP^1$. This symmetric product may be considered as $\PP(V)$ where $V$ is the 4-dimensional space of cubic polynomials $p(x)$ with roots $a_1,a_2,a_3$. If $x_1,\dots, x_6\in \PP^1$ are the images of the fixed points of $\sigma$  (so that the curve has equation $y^2=(x-x_1)\dots(x-x_6)$) then the branch locus consists of the six planes $p(x_i)=0$.

In \cite{H} we calculated the three-dimensional space of  sections of $\End_0 E\otimes K$ in terms of a $C^{\infty}$ splitting of the extension: writing $E=\OO\oplus L$ with $\bar\partial$-operator $\bar\partial_E(f,g)=\bar\partial f +\alpha g$ where $\alpha\in \Omega^{0,1}(C, L^*)$ represents the extension class. The condition $[\alpha]s=0$ means we can write $\alpha=\bar\partial u/s$ where $u$ is supported in a neighbourhood of the zeros $p,q,r$ of $s$ and in fact can be chosen to vanish at $q,r$.

In  concrete terms the space of Higgs fields has basis
$$(x-a_2)\begin{pmatrix} 1 &    \frac{2}{s}(u-u_1\frac{(x-a_3)}{(a_1-a_3)})\cr
0 &   -1\end{pmatrix}\frac{dx}{y},\quad
(x-a_3)\begin{pmatrix}1 &  \frac{2}{s}(u-u_1\frac{(x-a_2)}{(a_1-a_2)})\cr
0 &  -1\end{pmatrix}\frac{dx}{y}$$
$$\begin{pmatrix}-u &   \frac{1}{s}(-u^2+u_1^2\frac{(x-a_2)(x-a_3)}{(a_1-a_2)(a_1-a_3)}\cr
                               s &    u\end{pmatrix}\frac{dx}{y} $$ 
                               where $u_1=u(p)$.
\begin{remark}
Consider the map from $\OO\oplus K^*$ to $E$ given by $(1,t)\mapsto (1-ut,st)$ where $t$ is a local holomorphic section of $K^*$. Since $\bar\partial (1-ut)+(\bar\partial u/s)st=0$ this is  holomorphic and the quotient sheaf is $\OO_{p+q+r}(L)$. So a Hecke transform on $E$ at the points $p,q,r$ gives the bundle $\OO\oplus K^*$. But the formulas above show that each Higgs field $\Phi$ on $E$ preserves the trivial subbundle at $p,q,r$ and so they transform to the Higgs fields on $\OO\oplus K^*$. It follows that $H^0(C,\End_0E\otimes K)$, considered as a Lagrangian submanifold of the Higgs bundle moduli space ${\mathcal M}$, is a Hecke transform of the Hitchin section. Then the isomorphism of the multiplicity algebra with the cohomology of $(\CP^1)^3$ is a consequence of \cite{H1} Theorem 3.10, as well as the direct calculation in \cite{H}. 
\end{remark}

The quadratic form on $H^0(C,\End_0 E\otimes K)$ with values in $H^0(C,K^2)$ takes the form 
$$f(a_1)\frac{(x-a_2)(x-a_3)}{(a_1-a_2)(a_1-a_3)}\xi_1^2+f(a_2)\frac{(x-a_3)(x-a_1)}{(a_2-a_3)(a_2-a_1)}\xi_2^2+f(a_3)\frac{(x-a_1)(x-a_2)}{(a_3-a_1)(a_3-a_2)}\xi_3^2$$
where $f(x)=(x-x_1)\dots (x-x_6)$ and we identify $H^0(C,K^2)$ with  $\pi^*H^0(\PP^1,\OO(2))$. 

Our main focus in this article is the discriminant, and   since the multiplicity algebra has relations $\xi_1^2=\xi_2^2=\xi_3^2=0$ this  consists of three lines in $\PP^2$. However, introducing the geometry of the curve itself we have three lines in $\PP(H^1(C,K^*))\cong \PP^2$, the dual space of $H^0(C,K^2)$. The linear system $K^2$ maps $C$ to a conic $C_0$ in $\PP(H^1(C,K^*))$ -- evaluating $H^0(C,K^2)$ at a point $p\in C$ is the map.   Considering the coefficients of $\xi_i^2$ in the above formula we see that the three lines form the sides of the triangle with vertices $a_1,a_2,a_3$ on the conic, isomorphic to $\PP^1=\pi(C)$. Thus, apart from the double covering, the discriminant reproduces Atiyah's parametrization.

\subsection{Even degree} \label{even}
Now consider  $E$ a stable bundle over $C$ with $\Lambda^2E$ trivial. The moduli space of (S-equivalence classes of) semi-stable bundles in this case is well-known to be $\PP^3$ \cite{Nar0}, defined by considering the line bundles $L$ of degree $1$ such that $E\otimes L$ has a non-zero section: for a given $E$, $L$ varies in a  $2\Theta$-divisor in $\Pic^1(C)$. There are explicit formulas in \cite{Lor} for the  three quadratic functions: 
\begin{eqnarray*}
h_1&=& rst[\xi_0(u_0^2-1)+\xi_1(u_0u_1+u_2)+\xi_2(u_2u_0+u_1)]^2-\\
&& st[\xi_0(u_0u_1-u_2)+\xi_1(u^2_1+1)+\xi_2(u_1u_2+u_0)]^2+\\
&&4rs(\xi_0u_0+\xi_1u_1)^2-rt[\xi_0(u_0^2+1)+\xi_0(u_0u_1+u_2)+\xi_2(u_2u_0-u_1)]^2\\
h_2&=&t(u_0^2+u_1^2+u_2^2+1)[(\xi_0^2+\xi_1^2+\xi_2^2)+(\xi_0u_0+\xi_1u_1+\xi_2u_2)^2]+\\
&&st(u_0^2-u_1^2+u_2^2-1)[(\xi_0^2-\xi_1^2+\xi_2^2)-(\xi_0u_0+\xi_1u_1+\xi_2u_2)^2]+\\
&&4r(u_0u_2-u_1)[\xi_0\xi_2+(\xi_0u_0+\xi_1u_1+\xi_2u_2)\xi_1]+\\
&&4sr(u_2u_0+u_1)[\xi_2\xi_0-(\xi_0u_0+\xi_1u_1+\xi_2u_2)\xi_1]+\\
&&4s(u_1u_2+u_0)[\xi_1\xi_2-(\xi_0u_0+\xi_1u_1+\xi_2u_2)\xi_0]+\\
&&4rt(u_0u_1+u_2)[\xi_0\xi_1-(\xi_0u_0+\xi_1u_1+\xi_2u_2)\xi_2]\\
h_3&=&s[\xi_0(u_2u_0+u_1)+\xi_1(u_1u_2+u_0)+\xi_2(u_2^2-1)]^2-\\
&&[\xi_0(u_2u_0-u_1)+\xi_1(u_1u_2+u_0)+\xi_2(u_2^2+1)]^2-\\
&&t[\xi_0(u_0u_1+u_2)+\xi_2(u_1u_2-u_0)+\xi_1(u_2^2+1)]^2+4r(\xi_1u_1+\xi_2u_2)^2.
\end{eqnarray*}
Here $(u_0,u_1,u_2)$ are affine coordinates on $\PP^3$, and $(\xi_0,\xi_1,\xi_2)$ the corresponding coordinates on the cotangent space. The genus $2$ curve $C$ is the double cover of $\PP^2$ branched over the six points $0,1,\infty, r,s,t$. Then $E$ is determined by fixing $u_i$ and the relations in the algebra are given by the vanishing of the $h_i$.

It is difficult to draw conclusions about the structure of the multiplicity algebra from this mass of formulae so  here we adopt a more geometric approach, drawing on the paper of  van Geemen and Previato \cite{vGP}, which was the first investigation into explicit formulae for the integrable system. The genus 2/3/5 story below appears in various contexts  (see \cite{Bruin}, \cite{Duc}, \cite{Mas},\cite{Verra}).

If $E\otimes L$ has a non-zero section, then by Riemann-Roch, Serre duality and the isomorphism $E^*\cong E\otimes \Lambda^2E^*\cong E$,  we have $H^0(C,E\otimes KL^*)\ne 0$.  It follows that each $2\Theta$-divisor is symmetric with respect to the involution $L\mapsto KL^*$ on $\Pic^1(C)$. If the divisor is smooth, it is a curve $C_5$ of genus 5 and the involution acts freely with quotient $C_3$ of genus $3$ (we use the notation of \cite{vGP}). This is a plane section of a Kummer quartic surface, and if it is nonsingular then the bundle is certainly very stable, since, as in \cite{PP}, the complement -- the so-called ``wobbly" locus -- consists of  either the strictly semistable bundles or the 16 hyperplanes which meet the Kummer surface in a double conic. 

The double covering is defined by a line bundle $U$ on $C_3$ with $U^2$ trivial and we consider the 2-dimensional space $H^0(C_3,K_3U)$. The tangent space at $[E]\in {\mathcal N}$ is $H^1(C,\End_0E)$ but in terms of the corresponding $2\Theta$-divisor $C_5$ it is identified with a subspace of sections of the normal bundle in $\Pic^1(C)$. The normal bundle  is the canonical bundle so there is a  map from global sections of the tangent bundle of $\Pic^1(C)$ to sections of the normal bundle i.e. $H^1(C,\OO)\rightarrow H^0(C_5,K_5)$. Deformations of $E$ are given by curves $C_5$ which are divisors of the $2\Theta$ line bundle and so are transverse to translations on $\Pic^1(C)$, which is the image of $H^1(C,\OO)$. They are also symmetric with respect to the involution. It follows that $H^1(C,\End_0 E)$ is isomorphic to the even sections of $K_5$ and $H^1(C,\OO)$ to the odd ones. In terms of the genus $3$ curve we have 
\begin{equation} \label{isos}
H^1(C,\End_0 E)\cong H^0(C_3,K_3),\qquad H^1(C,\OO)\cong H^0(C_3,UK_3)
\end{equation}
 and taking duals $H^0(C,\End_0 E\otimes K)\cong H^1(C_3,\OO)$.
 
If $s,t$ form a basis of $H^0(C_3,UK_3)$ then we have sections $q_1=s^2,q_2=st,q_3=t^2$ of $K_3^2$. If $C_3$ is nonhyperelliptic then every quadratic differential $q_i$ is uniquely a quadratic form $Q_i$ in elements of $H^0(C_3,K_3)$. Since $q_1q_3=s^2t^2=(st)^2=q_2^2$ the homogeneous quartic equation $Q_1Q_3-Q_2^2=0$ defines $C_3$ in its canonical embedding: $C_3\subset \PP(H^1(C_3,\OO))$. 

The key result is: 
\begin{prp} \cite{vGP} \label{quad} Under the  isomorphism $H^0(C,\End_0 E\otimes K))\cong H^0(C_3,K)^*$ the net of conics is spanned by  $Q_1,Q_2,Q_3$. 
\end{prp} 
\begin{remark}

\noindent 1. The three quadratic forms correspond to basis elements of $S^2H^0(C_3,UK_3)\cong S^2H^1(C,\OO)$ and since each quadratic differential in genus 2 is a quadratic in sections of $K$, this is $H^1(C,K^*)\cong H^0(C,K^2)^*$. This is the invariant form for the map but it is more convenient from our point of view to think of the multiplicity algebra as the 8-dimensional algebra generated by $1,x,y,z$ with relations   $Q_i(x,y,z)=0$. 

\noindent 2. The essential input for the authors of \cite{vGP} to prove the isomorphisms in (\ref{isos})  involves the section $s$ of $E\otimes L$ and $s'$ of $E\otimes KL^*$. Then $\Phi=s\otimes s'-\langle s, s'\rangle/2$ is a trace zero Higgs field. 

\end{remark}

 The discriminant in $\PP^2=\PP(H^1(C,K^*))$ is the cubic curve $D$ defined by $\det (\sum_{i=1}^3y_iQ_i)=0$.  The curve $C$, as before, is mapped to a conic $C_0\subset \PP(H^1(C,K^*))$ by the linear system $K^2$   with six distinguished points $x_1,\dots, x_6$ the branch points of $\pi:C\rightarrow \PP^1$.  We shall use the following
 \begin{prp}\label{W}
 The six points $x_1,\dots, x_6$  lie on the discriminant. 
 \end{prp}
 \begin{proof}
  A fixed point $\tilde x_i$ of $\sigma$  is the divisor for the unique section $s$ of a square root $K^{1/2}$ of the canonical bundle. Hence $\dim H^0(C,K^{1/2})$ is odd.  If $E$ is a rank 2 bundle of degree zero, and hence topologically trivial,  then since $\End_0 E$ is odd-dimensional and has an orthogonal structure given by $\tr \phi^2$ it follows from the mod 2 index theorem \cite{AS} that $\dim H^0(C, \End_0 E\otimes K^{1/2})$ is also odd (this is part of a more general story \cite{Ox}) and hence contains a non-zero section $\Psi$. So  $s\Psi\in H^0(C,\End_0 E\otimes K)$ vanishes  at $\tilde x_i$ and hence so does $\tr(s\Psi\Phi)$ for all $\Phi$. This means that evaluation of $H^0(C,K^2)$ at  $\tilde x_i$, i.e.  its image of $x_i$ in $\PP(H^1(C,K^*))$ under the bicanonical map, gives a degenerate quadratic form, and hence lies on the discriminant. \end{proof} 
 
 \begin{remark} For higher genus $\dim H^0(C,\End_0\otimes K)=3g-3> 3=\rk (\End_0\otimes K)$ and so for any point $x\in C$ there exists a Higgs field vanishing at $x$. Consequently the bicanonical image of $C$ always lies in the discriminant hypersurface. 
  \end{remark} 

To describe the discriminant curve for a bundle $E$, from Proposition \ref{W}  we need to understand the cubics through the six points $x_i$ on $C_0$.
 Blow up $\PP^2$ at these points, then  the conic $C_0$ becomes a $-2$ curve. The linear system of cubics through the six points is then equivalent to $3H-E_1-\cdots-E_6$,  where $E_i$ are the exceptional curves and $H$ the class of a line. This  maps the surface  to $S\subset \PP^3$,  a cubic surface with a node from the collapse of the $-2$ curve.  Then any cubic curve through the $x_i$, $i=1,..,6$,  is defined by a plane section of $S$. We shall show next that conversely a generic section arises from a stable bundle $E$.

 \begin{prp} To each genus $2$ curve $C$ we associate a singular cubic surface $S$ as above. Then  a generic plane section is the discriminant of the net of conics defined by a very stable rank 2 bundle $E$ on $C$.
 \end{prp}
 \begin{proof}
We start with the classical fact (see \cite{Bea2}) that the choice of a non-trivial line bundle of order $2$ on a nonsingular plane cubic curve $D$ provides an equation of $D$  as $\det (\sum_{i=1}^3 y_iQ_i)=0$ where $Q_1,Q_2,Q_3$ are symmetric $3\times 3$ matrices. Parametrize the conic $C_0\cong \PP^1$ by $(y_1,y_2,y_3)=(1,2u,u^2)$ then the six points of intersection $b_i\in D\cap C_0$ are the roots of $\det (Q_1+2uQ_2+u^2Q_3)=0$.

Each $Q_i$ defines a conic $Q_i(x,y,z)=0$ in a projective plane $\PP^2$ and $Q_1Q_3-Q_2^2=0$ a quartic curve $C_3$ which for generic $D$ will be smooth. Since $\OO(1)$ on $\PP^2$ is the canonical bundle $K_3$ on $C_3$ the equation says that $Q_1$ on $C_3$ is a section of $K^2_3$ with double zeros, that is $s^2$ for a section $s$ of $K_3U$ for some line bundle $U$ with $U^2$ trivial. But on $C_3$ 
$$(uQ_1+vQ_2)^2=u^2Q_1^2+2uvQ_1Q_2+v^2Q_2^2=Q_1(u^2Q_1+2uvQ_2+v^2Q_3)$$
so as $u,v$ vary the bundle $U$ is constant and we have $Q_1=s^2, Q_3=t^2, Q_2=st$ for a basis $s,t$ of sections of $K_3U$. 

The line bundle $U$ defines an unramified covering, a curve $C_5$ of genus $5$ with an involution $\tau$. The Prym variety $\PP(C_5,C_3)$ is a 2-dimensional abelian variety consisting of the divisor classes in $\Pic^0(C_5)$ which are anti-invariant under $\tau$. It has two components and $x-\tau(x)$ embeds $C_5$ (if it is not hyperelliptic) in the non-trivial component. Then (see \cite{Verra} or the other references above) each component is isomorphic to the Jacobian of the genus 2 curve $y^2=\det(Q_1+2uQ_2+u^2Q_3)$ and $C_5$ is a symmetric $2\Theta$-divisor with $C_3$ a plane section of the Kummer surface. This defines a point in the moduli space which, thanks to Proposition \ref{quad}, has the given cubic $D$ as discriminant. 

The isomorphism of the two abelian surfaces can be standardized up to elements of order two \cite{vGP}, for at the branch point $x_i$ the conic is a pair of lines each of which is a bitangent to $C_3$, that is sections $s_1, s_2$ of $K_3^{1/2}, K_3^{1/2}U$. Then pulled back to $C_5$ this  gives  a line bundle $K^{1/2}_5$ with $\dim H^0(C_5, K_5^{1/2})=2$. So $b_i\in \Pic^1(C)$ corresponds to $K_5/2\in \Pic^4(C_5)$ and $K_5/2+a-\tau(a)$ identifies with the Prym variety. Since we are concerned with $\End_0 E\otimes K$, the choice of line bundle of order $2$ is irrelevant
\end{proof} 
\begin{corollary} For a fixed genus $2$ curve, the isomorphism class of the multiplicity algebra for a very stable rank 2  bundle $E$ of even degree varies continuously as $E$ varies. 
\end{corollary}
\begin{proof}
From  \cite{Bea1}, apart from quadrics, smooth plane sections of a surface vary  the modulus of the curve non-trivially. In our case each section is a cubic curve isomorphic to the discriminant of the  net of conics associated to $E$. 
\end{proof} 

\begin{remark}
 Passing from the discriminant to the bundle $\End_0 E$ involved a degree 3 covering of the space $\PP^3$ of plane sections of the singular cubic surface. This fits into a classical situation:  two  threefolds, dual to each other, the Igusa quartic $B$ in $\PP^4$ and the Segre cubic in the dual projective space.  The quartic has an interpretation as  the GIT quotient of six points on a conic and a point $x$ on it represents by duality a hyperplane  section of the cubic. The six points $x_i$ give us the genus 2 curve $C$ and the singular cubic surface $S$ is  constructed, as in the proposition, by blowing up the points. Then the plane sections of $S$ correspond by duality to the lines through $x$. Projection from $x\in B$ is then a threefold covering of $\PP^3$, the space of discriminants. But the Igusa quartic is also the quotient of $\PP^3$ by  the action of $H^1(C,\Z_2)$, or equivalently the moduli space of projective bundles $\PP(E)$, so we have an analogue of Atiyah's description of the moduli space as a branched cover of $\PP^3$, at least in the very stable situation. 
\end{remark}

\section{Genus 3 curves}\label{three} 
\subsection{The vector bundles}\label{vect}
Let $C$ be a nonhyperelliptic curve of genus $3$ -- a nonsingular plane quartic in the canonical embedding. The moduli space ${\mathcal N}$ of (semi)stable rank two bundles $E$ with $\Lambda^2E$ trivial  is isomorphic to  the Coble quartic  hypersurface in $\PP^7$   \cite{Nar1}. The decomposable bundles $E=L\oplus L^*$ describe a three dimensional Kummer variety, the quotient of  the Jacobian $\Jac(C)$ by $x\mapsto -x$, which is the singular locus. The  Coble quartic is characterized by this property together with invariance by the action  on $\C^8$ of a finite Heisenberg group which corresponds to $E\mapsto E\otimes U$ for $U\in H^1(C,\Z_2)$, a line bundle with $U^2$ trivial. 

The discriminant here is a degree 6 hypersurface in $\PP(H^0(C, K^*))=\PP^5$ which is quite complicated -- as noted above the bicanonical embedding of $C$ lies in it, and the rank of the quadratic form there is at most 3, so this is a curve of singularities. We shall consider instead the family of bundles which are defined by  the fixed points in ${\mathcal N}$ of one element $U$. This is the intersection of a $\PP^3\subset \PP^7$ with the quartic hypersurface, hence a quartic surface, and we have a 2-dimensional family. 

A fixed point in the moduli space means an isomorphism $\psi:E\rightarrow E\otimes U$. We can view this as a Higgs bundle twisted with $U$ rather than the usual $K$, which implies that we have a genus 5 spectral curve $\pi:\tilde C\rightarrow C$, the unramified covering corresponding to $U$, and $E$ is the direct image $\pi_*(L\otimes \pi^*U^{1/2})$ where $\sigma^*L\cong L^*$ i.e. $L$ lies in the Prym variety. The bundle $E$ is strictly semistable if $L^2$ is trivial. 

 A section $e$ of $E=\pi_*(L\otimes \pi^*U^{1/2})$ on an open set $V\subset C$ is by definition of the direct image a section $s\in H^0(\pi^{-1}(V),L\otimes \pi^*U^{1/2})$ and then $s\sigma^*s\in H^0(V,U)$ is a nondegenerate $U$-valued quadratic form $(e,e)$ on $E$.  In the presence of this orthogonal structure, the trace zero endomorphisms $\End_0E$ split into  symmetric and skew symmetric components, so we have 
\begin{equation}\label{sum}
\End_0 E\cong U\oplus E'.
\end{equation}
The isomorphism $\psi:E\cong E\otimes U$ gives an involution on $\End_0 E$ and the above summands are the eigenspaces. Moreover 
the rank 2 bundle  $E'$ has an orthogonal structure, this time with values in the trivial bundle.  If we consider $\psi$ as a $U$-valued Higgs field on $E$, its eigenspaces in the adjoint representation  are squares of the eigenspaces on $E$. This means that $E'=\pi_*L^2$. 
 
The expression $\End_0 E\cong U\oplus E'$ is an orthogonal decomposition with respect to the   quadratic form $\tr \phi^2$ so the  map $$\tr \Phi^2:H^0(C,\End_0E\otimes K)\rightarrow  H^0(C, K^2)$$ in this case takes $(u,e)\in H^0(C,UK) \oplus H^0(C,E'\otimes K)$ to a multiple of $u^2+(e,e)$. 

We met  the first term $H^0(C,UK)$ with $C=C_3$ in the previous section, giving quadratic forms $Q_1,Q_2,Q_3$. These will appear next in a different form.

 \subsection{The net of quadrics}
In \cite{P},  Pauly associates to a rank 2 stable bundle $E$ with trivial determinant a bundle $F$ with $\Lambda^2F\cong K$ such that $\dim H^0(C,E\otimes F)=4$. It arises as follows: $E$ is defined as an extension 
$L^*\rightarrow E \rightarrow L$
where $\deg L=1$ and the extension class lies in  $ H^1(C,L^{-2})$ which has dimension $4$. This class annihilates a 3-dimensional subspace $W\subset H^0(C,L^2K)$. Then $(F\otimes L)^*$ is defined as the kernel of the evaluation map ${\mathrm{ev}}: C\times W\rightarrow L^2K$. 
In the long exact sequence 
\begin{equation} \label{long}
0\rightarrow H^0(C,F\otimes L^*)\rightarrow H^0(C,F\otimes E)\rightarrow H^0(C, F\otimes L) \stackrel{\delta} \rightarrow H^1(C,F\otimes L^*)
\end{equation}
consider first 
$H^0(C,F\otimes L^*)$. We have $(F\otimes L)^*\subset C\times W$ so since $F\cong F^*\otimes K$,  $F\otimes L^*K^*\subset C\times W$. Tensoring with $K$ gives 
$$H^0(C,F\otimes L^*)\rightarrow H^0(C,K)\otimes W\rightarrow H^0(C, K^2L^2)$$
and by Riemann-Roch $\dim H^0(C,F\otimes L^*)\ge 9-8=1$. 
The connecting homomorphism $\delta$ in the sequence (\ref{long}) is zero on $W^*\subset H^0(C, F\otimes L)$ and, if the bundle is not a so-called exceptional one, $\dim H^0(C,F\otimes L^*)=1$ and this gives  the four dimensions for $H^0(C,E\otimes F)$.
Importantly, it turns out that $F$ is independent of the choice of $L^*$, which is a maximal subbundle -- there are generically eight of these.
As shown in \cite{P}, if $E$ is stable then $F$ is stable unless $E$ is defined by a  point on the Coble quartic which lies on a trisecant  of the Kummer variety. 

The isomorphism  $\Lambda^2E\cong \OO$  defines a skew form on $E$ and $\Lambda^2F\cong K$ gives a  skew form on $F$ with values in $K$ so the tensor product  has a $K$-valued {\it symmetric} form.  Then the  4-dimensional space $V=H^0(C,E\otimes F)$  has a symmetric bilinear form $(\,,\,)$ with values in $H^0(C,K)$, geometrically  a  net of quadrics in $\PP^3=\PP(V)$. 

\subsection{The multiplicity algebra}\label{multialg}

Following on from this approach, in \cite{H} the author made use of 
the natural map 
$$\Lambda^2 (E\otimes F)\rightarrow S^2E\otimes \Lambda^2F\cong \End_0E\otimes K$$
which induces one from $\Lambda^2 H^0(C,E\otimes F)$ to $H^0(C, \End_0E\otimes K)$. Both spaces have dimension 6 and it was shown that this is an isomorphism. 

Explicitly, if 
we choose a basis $v_1,\dots, v_4$ of $V$ then the quadratic form $\tr \Phi_1\Phi_2$ with values in  $H^0(C,K^2)$ is given up to a scale by 
\begin{align*}
(v_i\wedge v_j,v_i\wedge v_k)&=(v_i,v_i)(v_j,v_k)-(v_i,v_k)(v_j,v_i)\\
(v_1\wedge v_2, v_3\wedge v_4)&=(v_1,v_3)(v_2,v_4)-(v_1,v_4)(v_2,v_3)+\sqrt{\det(v_i,v_j)}.
\end{align*}
The square root means that the equation of the quartic is $\det(v_i,v_j)-Q^2$ for some quadratic $Q(x,y,z)$. Thus $\det(v_i,v_j)=0$ is a quartic curve meeting $C$ tangentially at 8 points, and Pauly describes in more detail the rational map from the moduli space of projective bundles $\PP(E)$ to the space of tangential quartics.

In our case $E\cong E\otimes U$ and from the construction of the bundle $F$ it follows that $F\cong F\otimes U$ which means the composition of the isomorphisms defines an involution $\tau$ on $V=H^0(C,E\otimes F)$ which is orthogonal with respect to the $H^0(C,K)$-valued inner product. We also have the symmetric form on $E$ with values in $U$ which gives a skew form on $V$ with values in $H^0(C,UK)$ (this defines a homomorphism from $\Lambda^2(E\otimes F)$ to $UK$ which gives the decomposition (\ref{sum})). 

\begin{remark} The analogy was drawn in \cite{H} with the differential geometry of four dimensions expressed in terms of the two spinor bundles $S_+,S_-$ with $E,F$ playing similar roles here. In pursuit of this analogy the situation we have here is parallel to K\"ahler geometry in two complex dimensions
with the involution being parallel to the complex structure and the skew form to the K\"ahler form. 
\end{remark} 

The involution $\tau$ on $V$ induces one on $\Lambda^2V\cong H^0(C,\End_0V\otimes K)\cong H^0(C,UK)\oplus H^0(C, E'\otimes K)$. The first summand is a 2-dimensional $+1$ eigenspace, with  a 4-dimensional $-1$ eigenspace as orthogonal  complement. 
It follows that $\tau$ splits $V=V^+\oplus V^-$ into two orthogonal 2-dimensional subspaces spanned by $v_1,v_2$ and $v_3,v_4$, the $\pm 1$ eigenspaces of $\tau$, and each has a quadratic form $A^+,A^-$ . Then $\det (v_i,v_j)=\det A^+\det A^-$ so the equation of the quartic is $\det A^+\det A^- - Q^2=0$ -- the familiar expression $Q_1Q_2=Q^2$ associated to $U$ as in the first section. 

In Pauly's description of the moduli space of projective bundles on $C$ using tangential quartic curves, the case we are considering is where the tangential quartic is a pair of conics, defined by two sections of $KU$. 

\subsection{Explicit forms} \label{explicit}
Take a basis $v_1,v_2,v_3,v_4$ as above and write $b_{ij}=(v_i,v_j)$,  sections of $K$,  and take a corresponding  basis $v_{23},v_{31},v_{12},v_{41},v_{42},v_{43}$ where $v_{ij}=v_i\wedge v_j$   for the six-dimensional exterior product $\Lambda^2V\cong H^0(C,\End_0 E\otimes K)$.   
The involution acts as $1$ on $v_{12},v_{34}$ and $-1$ on the others.  Then the matrix of the quadratic form $\tr \Phi^2$ with values in $H^0(C,K^2)$ 
 is 
\begin{equation}\label{mat} 
\begin{bmatrix} b_{22}b_{33} & -b_{12}b_{33} & 0 &Q -b_{12} b_{34} & -b_{22}b_{34} & 0\\
-b_{12}b_{33} & b_{11}b_{33} & 0 &  b_{11}b_{34} & Q+b_{12}b_{34} & 0\\
0 & 0 & Q_1 & 0 & 0 &Q \cr
Q -b_{12} b_{34}&b_{11}  b_{34}& 0 &b_{11}  b_{44}&  b_{12}b_{44} & 0 \\
-b_{22}b_{34} &Q+ b_{12}b_{34}& 0 & b_{12}b_{44} & b_{22} b_{44}& 0\\
0 & 0 & Q & 0 & 0 & Q_2
\end{bmatrix}
\end{equation} 
where $Q_1=\det A^+=b_{11}b_{22}-b_{12}^2, Q_2=\det A^-=b_{33}b_{44}-b_{34}^2$.  The discriminant is given by the vanishing of this determinant.


\begin{remark}To explain the  roles of quadratic forms in this picture,  recall that the discriminant lies in $\PP(H^1(C,K^*))$ and by Serre duality $H^1(C,K^*)$ consists of linear functions on $H^0(C,K^2)$ and the latter are quadratic expressions in $x,y,z\in H^0(C,K)$. Then $x^2,y^2, z^2, xy \dots$  as sections of $H^0(C,K^2)$ are to be regarded as linear functions on its dual  $H^1(C,K^*)$. 
\end{remark}

It is clear from $(\ref{mat})$ that $Q_1Q_2-Q^2$ is a quadratic factor of the determinant and this gives a singular quadric whose degeneracy subspace is the $\PP^2$ defined by the linear forms $Q_1=Q_2=Q=0$. The other component is a quartic hypersurface $X$. Then $X\cap \PP^2$ is a quartic curve whose modulus is an invariant of the algebra. 

We look at this slightly differently. 
The intersection of  the hypersurface $X$ with the $\PP^3$ defined by $Q_1=Q_2=0$ is a surface given by the determinant of the $4\times 4$ matrix obtained from (\ref{mat}) by deleting rows and columns 3 and 6. Its intersection with $Q=0$ is the determinant of this matrix where we set $Q=0$ also. This is 
\begin{equation}\label{mat1} 
\begin{bmatrix} b_{22}b_{33} & -b_{12}b_{33}  & -b_{12} b_{34} & -b_{22}b_{34} \\
-b_{12}b_{33} & b_{11}b_{33}  &  b_{11}b_{34} & b_{12}b_{34} \\
-b_{12} b_{34}&b_{11}  b_{34} &b_{11}  b_{44}&  b_{12}b_{44}  \\
-b_{22}b_{34} & b_{12}b_{34} & b_{12}b_{44} & b_{22} b_{44}\\
\end{bmatrix}.
\end{equation} 
which now depends only on the two conics  and not on the genus 3 curve $C$. It is the inner product on the 4-dimensional anti-invariant part of $\Lambda^2V$ induced from the inner product on $V$. 
Its determinant defines a quartic surface $S\subset \PP^3$ and the curve $X\cap \PP^2$ we are seeking is its intersection with the plane $Q=0$.

\begin{remark} The determinant of this is to be understood as a  polynomial in the quadratics in $x,y,z$ without imposing the relations $(xy)(yz)=(xz)(y^2)$ etc. Those relations define a Veronese surface in $\PP^5$ and the determinant is then $(Q_1Q_2)^2$. 
\end{remark}

To obtain a formula  we take a generic pair $Q_1=x^2-y^2-z^2,  Q_2=a^2x^2-b^2y^2-z^2$  or equivalently we take the quadratic form on $V\cong \C^4$ to be 
\begin{equation}
\begin{bmatrix} A^+ & 0\\
0 & A^-
\end{bmatrix}=
\begin{bmatrix} x+y & z & 0 & 0\\
                              z & x-y & 0 & 0\\
                              0 & 0 & ax+by & z\\
                              0 & 0 & z & ax-by
                              \end{bmatrix} \label{4mat}
                              \end{equation} 
From this point of view the six parameters in the choice of $Q(x,y,z)$ provide the $3g-3=6$ degrees of freedom for the genus 3 curve $C$:  $(x^2-y^2-z^2) (a^2x^2-b^2y^2-z^2)=Q^2$. The extra parameters 
$a,b$ yield the 2-dimensional family of vector bundles but, from the  equation of the curve, they are interrelated.

The $\PP^3$ is given by $Q_1=0=Q_2$ so we use the relations  $x^2-y^2-z^2=0, a^2x^2-b^2y^2-z^2=0$ to give homogeneous coordinates  $w_1=yz, w_2=zx, w_3=xy, z^2=(a^2-b^2)w_0$  and then the matrix (\ref{mat1}) is (with $c=(a-b)(1+ab), d=(a+b)(1-ab)$)
 
\begin{equation}\label{mat2}
M=\begin{bmatrix}
cw_0-(a-b)w_3 & -bw_1-aw_2 & (b^2-a^2)w_0 & w_1-w_2\\
-bw_1-aw_2 & dw_0 + (a+b)w_3 & w_1+w_2 & (a^2-b^2) w_0\\
(b^2-a^2) w_0& w_1+w_2 & cw_0+(a-b)w_3 & -bw_1+aw_2\\
w_1-w_2 & (a^2-b^2)w_0 & -bw_1+aw_2 & d w_0-(a+b) w_3
\end{bmatrix} 
\end{equation} 
Incorporating terms from the third and sixth rows and columns of (\ref{mat}) this yields the following explicit formulas for the algebra. It is 
 generated by $1, \xi_1,\xi_2,\xi_3,\eta_1,\eta_2,\eta_3$ subject to the following  relations, where $\xi\mathbf{\cdot} \eta=\xi_1\eta_1+\xi_2\eta_2+\xi_3\eta_3$: 
 \newpage
$$(a-b)(1+ab)(\xi_1^2+\eta_1^2)+(a+b)(1-ab)(\xi_2^2+\eta_2^2) +2(a^2-b^2)(\xi_1\eta_1-\xi_2\eta_2)+ a_1 \xi\mathbf{\cdot} \eta$$
$(\xi_1\eta_2-\xi_2\eta_1)+a(\xi_1\xi_2-\eta_1\eta_2)+a_2 \xi\mathbf{\cdot} \eta,\qquad  
(\xi_1\eta_2+\xi_2\eta_1)-b(\xi_1\xi_2+\eta_1\eta_2)+a_3 \xi\mathbf{\cdot} \eta$
$$(a+b)(\xi_2^2-\eta_2^2) -(a-b)(\xi_1^2-\eta_1^2)+ a_4 \xi\mathbf{\cdot} \eta, 
\qquad\xi_3^2+a_5 \xi\mathbf{\cdot} \eta,\qquad
\eta_3^2+a_6 \xi\mathbf{\cdot} \eta.$$
Just as the formulas for genus 2 in Section \ref{even} were of little use in determining the isomorphism class of the algebra,   in this case we consider a more degenerate version and observe that our family consists of a continuous  deformation of  it.

\subsection{A special case}\label{special} 

The first algebra to be calculated using this method is in \cite{H} where the relations  have the simple form
\begin{equation}\label{rel0}
\xi_i^2=a_i(\xi_1\eta_1+\xi_2\eta_2+\xi_3\eta_3),\qquad \eta_i^2=b_i(\xi_1\eta_1+\xi_2\eta_2+\xi_3\eta_3)
\end{equation} 
This was based on taking $Q_1=xy, Q_2=z(x+y+z)$ two degenerate quadratic forms. The discriminant is then given by three singular quadrics and the algebra  may be regarded as a deformation of the relations $\xi_i^2=\eta_i^2=0$ for the cohomology of $H^*((\CP^1)^6,\C)$. This is a continuous variation in the isomorphism class of the algebra, but the parameters describe a variation of the curve $C$ rather than the bundle $E$. In fact, if we take the approach of Section \ref{vect} then
the bundle $E'$ here corresponds to taking $L^2=\pi^*U_1$  where $U_1^2$ is trivial, hence for a given curve there are only finitely many bundles of this type.

In this case we have $\End_0 E\cong U\oplus UU_1\oplus U_1$. We can check if $E$  is very stable by looking for base points of the family of quadrics, that is non-trivial solutions to (\ref{rel0}), but these only exist if  
$\sqrt{a_1b_1}+\sqrt{a_2b_2}+\sqrt{a_3b_3}=1$. So for a generic curve $C$ a bundle of this type is very stable. 

Our family of vector bundles is obtained by direct image from the Prym variety by deforming $L$ away from $U_1$ and so we can say already that a generic member will be very stable. 
\begin{remark} In \cite{PP} the authors show that the ``wobbly" bundles are cut out by a hypersurface in $\PP^7$ of degree $48$.
\end{remark} 

For this special case the quartic curve $X\cap \PP^2$ is a pair of conics. To show a variation in the modulus it is enough to show that as $a,b$ vary we get an irreducible curve. But by analyzing the surface $S$ we get rather more. 

\section{The quartic surface $S$}
\subsection{Plane sections} 
A variety such as $S$ ($\det M=0$ from (\ref{mat2})) which is defined by the determinant of a {\it symmetric} matrix of linear forms in $n$ variables  has singularities when $n>3$.  These  hypersurfaces are classically known as symmetroids (see e.g. \cite{Dol}). For a quartic surface the generic symmetroid has  10 nodes. 

To find the singular locus in our case first note that the plane section $w_0=0$ gives a quartic curve
$$(b^2-1)^2w_1^4+((a^2-1)w_2^2+(a^2-b^2)w_3^2)^2-2(b^2-1)w_1^2((a^2-1)w_2^2+(b^2-a^2)w_3^2)=0$$
which is reducible to a pair of conics:
$$(b^2-1)w_1^2=(a^2-1)w_2^2-(a^2-b^2)w_3^2\pm 2\sqrt{(1-a^2)(a^2-b^2)}w_2w_3.$$
These meet where $w_2=0$ or $w_3=0$.

If $w_0=w_3=0$ $\sqrt{b^2-1}w_1=\pm\sqrt{a^2-1}w_2$ is a singular point of this curve. The matrix is then 
$$\begin{bmatrix} 0 & -bw_1-aw_2 & 0 & w_1-w_2\\
-bw_1-aw_2 & 0 & w_1+w_2 & 0\\
0 & w_1+w_2 & 0 & -bw_1+aw_2 \\
w_1-w_2 & 0 & -b w_1+aw_2 & 0
\end{bmatrix}$$
and two pairs of rows are linearly dependent so the matrix has rank $\le 2$. This is a singular point of $S$. 

Calculating the Hessian of the quartic function $\det M$ at this point one sees that unless $a^2=b^2$ or $(a^2-1)(b^2-1)=1$ the singularity is a node, and similarly at the other point $x_0=x_2=0$. Moreover, at no other points on the curve $w_0=0$ is the rank of the matrix $\le 2$, so this  means that $S$ has no codimension 1 singularities. Then a general plane section is a smooth quartic curve. 

\begin{prp} Let ${\mathcal N}$ be the moduli space of  rank 2 semi-stable bundles $E$ with trivial determinant on a non-hyperelliptic genus 3 curve $C$ and $U$ a line bundle with $U^2$ trivial.  Then a generic bundle with $E\cong E\otimes U$ is very stable,  and the isomorphism class of the multiplicity algebra varies non-trivially in this two-dimensional family, for a generic curve $C$. 
\end{prp}
 \begin{proof}
In Section \ref{explicit} we showed using \cite{P} that such a bundle $E$ is defined up to tensoring by a line bundle by a pair of tangential conics. A generic pair can be simultaneously diagonalized and hence defined by equations $Q_1(x,y,z)=x^2-y^2-z^2=0, Q_2(x,y,z)=a^2x^2-b^2y^2-z^2=0$. We calculated the discriminant of the family of quadrics to be a singular quadric and a quartic hypersurface in $\PP^5$.  The projective equivalence class of the curve of intersection of the degeneracy plane of the quadric with the quartic is an invariant of the isomorphism class of the multiplicity algebra. 

We identified the quartic curve as the intersection of a quartic surface $S$ with a hyperplane in $\PP^5$ defined by the conic $Q$ which gives $Q_1Q_2-Q^2=0$ as the equation of $C$. A generic section is a smooth quartic curve, but the special bundles 
 in  Section \ref{special}, which are  very stable for generic $C$, belong to the connected family of bundles $E$ considered. This family thus defines quartic curves which are both smooth and reducible and hence vary in modulus.
\end{proof} 
\subsection{ Ten singularities} 
Going back to the quartic surface $S$, by rescaling, the roles of $x,y,z$ interchange,  and since  $w_0$ corresponds to the $z^2$ term and $w_3$ to $xy$, we obtain nodes  by the method above for the pairs $(y^2, zx)$ and $(x^2, yz)$. This yields six singular points. 

To find the other four for the symmetroid, consider the intersection of the $\PP^3$ given by $Q_1=0=Q_2$ with the Veronese surface. As an element of the dual space of quadratic polynomials in $x,y,z$, a point on the Veronese corresponds to  evaluation at some point $(a,b,c)\in\C^3$. So the intersection with $\PP^3$ is in this case the four points of intersection $a_1,a_2,a_3,a_4$ of the conics $x^2-y^2-z^2=0, a^2x^2-b^2y^2-z^2=0$. In homogeneous coordinates these are $(x,y,z)=(\pm\sqrt{1-b^2},\pm \sqrt{1-a^2},\pm \sqrt{a^2-b^2})$. The symmetric matrix on $V$ given by (\ref{4mat}) then has rank $2$ and correspondingly  the inner product on the anti-invariant part of $\Lambda^2V$ (which is the matrix (\ref{mat2})) has rank  $\le 2$. This is thus a singular point of the determinantal surface $S$. 

The line in $\PP^3$ joining evaluation at $a_1=(\sqrt{1-b^2}, \sqrt{1-a^2}, \sqrt{a^2-b^2})$ to evaluation at $a_2=(\sqrt{1-b^2}, \sqrt{1-a^2}, -\sqrt{a^2-b^2})$ meets the plane $z^2=0$ in the singularity evaluated in the previous section, so we see that the ten nodes are labelled by the four points of intersection of two conics and the six lines joining them.

\section*{Acknowledgments}
The author wishes to thank Tam\'as Hausel for raising this question and ICMAT for support,  and in particular Oscar Garc\'ia-Prada for his continuing contributions to research into the geometry of vector bundles on curves.

\vskip 1cm
\centerline{Mathematical Institute, Woodstock Rd, Oxford OX2 6GG, UK}

\end{document}